\documentclass[review]{elsarticle}
\usepackage{}
\usepackage{amsfonts}
\usepackage{amsmath}
\usepackage{graphicx,setspace}
\usepackage{amsfonts,amsmath, amssymb, amsthm}
\usepackage{mathrsfs}
\usepackage{epstopdf}
\usepackage{float}
\usepackage{lineno,hyperref}
\usepackage{color}
\usepackage{subfigure}
\usepackage{bm}
\usepackage{graphicx}
\usepackage{amssymb, amsthm}
\usepackage{mathrsfs}
\usepackage{epstopdf}
\usepackage{float}
\usepackage{lineno,hyperref}
\usepackage{color}
\usepackage{subfigure}
\usepackage{cases}

\numberwithin{equation}{section}

\journal{Multiscale Modeling and Simulation}

\begin{document}

\newtheorem{definition}{Definition}
\newtheorem{lemma}{Lemma}
\newtheorem{remark}{Remark}
\newtheorem{theorem}{Theorem}
\newtheorem{proposition}{Proposition}
\newtheorem{assumption}{Assumption}
\newtheorem{example}{Example}
\newtheorem{corollary}{Corollary}
\def\e{\varepsilon}
\def\Rn{\mathbb{R}^{n}}
\def\Rm{\mathbb{R}^{m}}
\def\E{\mathbb{E}}
\def\hte{\bar\theta}
\def\cC{{\mathcal C}}
\numberwithin{equation}{section}

\begin{frontmatter}

\title{{\bf  Analysis of multiscale methods for stochastic dynamical systems driven by  $\alpha$-stable processes
}}
\author{\centerline{\bf Yanjie Zhang $^{a,}\footnote{zhangyj18@scut.edu.cn}$,
Xiao Wang$^{b,*}\footnote{corresponding author: xwang@vip.henu.edu.cn}$,
Zibo Wang$^{c} \footnote{ zibowang@hust.edu.cn }$
and Jinqiao Duan$^{d}\footnote{duan@iit.edu}$}
\centerline{${}^a$ School of Mathematics }
\centerline{South China University of Technology, Guangzhou 510641,  China}
\centerline{${}^b$ School of Mathematics and Statistics,} \centerline{Henan University, Kaifeng 475001, China}
\centerline{${}^c$ Center for Mathematical Sciences }
\centerline{Huazhong University of Science and Technology, Wuhan 430074,  China}
\centerline{${}^d$ Department of Applied Mathematics,} \centerline{Illinois Institute of Technology, Chicago, IL 60616, USA}}

\begin{abstract}
In this paper, we first analyze the strong and weak convergence of projective integration methods for multiscale stochastic dynamical systems  driven by $\alpha$-stable processes, which are used to estimate the effect that the fast components have on slow ones. Then we obtain the $p$th moment error bounds between the solution of slow component produced by projective integration method and the solution of effective system with $p \in \left(1,  \alpha\right)$. Finally, we corroborate our analytical results through a specific numerical example.
\end{abstract}


\end{frontmatter}


\section{Introduction}
The multiscale models arise widely in various fields \cite{rb,ea,DP, EH}. For example, the production of mRNA and proteins occur in a bursty, unpredictable, and intermittent manner, which create variation or noise in individual cells or cell-to-cell interactions. Since the mRNA synthesis process is faster than the protein dynamics, this leads to a multiscale system. Finding a coarse-grained model that can effectively describe the dynamics of the multiscale model has always been a very active research field. Khasminskii et al.\cite{RZ} developed a stochastic averaging principle driven by Wiener noise that enables one to average out the fast-varying variables. The main idea is as follows: under appropriate conditions, with the slow-varying component fixed, if the fast-varying component has a stationary distribution, it can be shown that the process represented by the
slow-changing component converges weakly to a limit averaging system.   Motivated by the previous works,  averaging principle for various stochastic dynamical systems or stochastic partial differential equations driven by Wiener noise have also drawn much attention, see, e.g., \cite{YK, SC1,SC2,HF,WW,GA}.  Some authors also studied the averaging principle of two-scale dynamical systems driven by non-Gaussian noises with finite second moments \cite{DL12, JX, DG}. This excludes the $\alpha$-stable noise, since its second moment is divergent \cite{da}.

 Recently, multiscale dynamical systems driven by $\alpha$-stable processes have drawn much attention. Bao et al. \cite{bg} studied the averaging principle for stochastic partial differential equation with two-time-scale Markov switching. They showed that under suitable conditions, a limit process that was a solution of either an SPDE or an SPDE with switching was obtained.  In \cite{YZ} and \cite{YZ1}, they studied data assimilation and parameter estimation and showed that the averaged, low dimensional filter approximated the original filter, by examining the corresponding Zakai equations. Sun et al. \cite{XS, XS1} studied the averaging principle for stochastic real Ginzburg-Landau equation and stochastic differential equation. They used the classical Khasminskii approach to show the convergence between the slow component and averaged equation. Moreover, they  also studied the  strong and weak convergence rates for slow-fast stochastic differential equations and proved that the strong and weak convergent order are $1-1/\alpha$  and $1$ respectively.

However, it is often impractical to obtain the reduce equations in closed form, since the invariant measure is often unknown. Standard computational schemes may fails due to the separation between the $O(\varepsilon)$ time scale and the $O(1)$. This inspires us to develop a new algorithms to estimate the effect that the fast components have on slow ones. Several related techniques  have been proposed for multiscale stochastic dynamics driven by Wiener noises or non-Gaussian noises with finite second moments.  The  heterogeneous multi-scale method (HMM) is a general methodology for efficient numerical computation of problems with multiple scales and/or multi-levels of physics. For example, E. Vanden-Eijnden \cite{VE} used   the HMM to compute the evolution of the slow variables without having to derive explicitly the effective equations beforehand. W. E et al. \cite{EW} analyzed a class of numerical schemes for the multiscale dynamical systems driven by Wiener noises. A similar idea, also called ``projective integration '' method (PIM) was proposed in \cite{DG06}. D. Givon et al. used this method to analyze  multiscale stochastic dynamics driven by noises with finite second moments and obtained explicit bounds for the discrepancy between the results of the PIM and the slow components of the original system, which excludes the very important $\alpha$-stable noise.  A natural and important question is the following: for the multiscale dynamical systems driven by $\alpha$-stable noises, how to estimate the effect that the fast components have on the slow ones as the invariant measure is unknown from the perspective of computation ?

The main  technique used in this present manuscript is the framework of  ``projective integration'' method, which consist of a hybridization between a standard solver for the slow components, and short runs for the fast dynamics. The main difficulty is how to deal with the nonlinear term and $\alpha$-stable process.

This paper is organized as follows. In Section 2,  we recall the basic concepts about symmetric $ \alpha$-stable process and ergodic theory. In Section 3,  we formulate the problem and give the strong convergence analysis of the projective integration method. In Section 4, we give the weak convergence analysis of the projective integration method. In section 5, we corroborate our analytical results through a specific numerical example.
Some discussions are contained in Section 6.

To end this section, we introduction some notations, $C$ with or without subscripts will denote a positive constant, whose value may change from one place to another. We will use $\langle \cdot, \cdot \rangle $  to denote the scalar product in $ \mathbb{R}^n $ and  $||\cdot||$ to denote the norm.  $\mathcal{B}_{b}(\mathbb{R}^{d})$ denotes the space of all Borel measurable functions. For any $k \in \mathbb{N}_{+}$ and $\delta \in (0,1)$, we define
\begin{equation*}
\begin{aligned}
&C^{k}(\mathbb{R}^{n}):=\left\{u: \mathbb{R}^{n}\rightarrow \mathbb{R}:\text{ u and all its partial  derivative up to order k are continuous} \right\},\\
&C^{k}_{b}(\mathbb{R}^{n}):=\left\{u \in C^{k}(\mathbb{R}^{n}):\text{ for  $1 \leq i \leq k$, the $i$ order partial  derivative are bounded } \right\},\\
&C^{k+\delta}_{b}(\mathbb{R}^{n}):=\left\{u \in C^{k}_b(\mathbb{R}^{n}):\text{ all the $k$-th order partial derivative of u are $\delta$- H\"older continuous} \right\}.
\end{aligned}
\end{equation*}
For $k_1, k_2 \in \mathbb{N_{+}}$, $0 \leq \delta_1, \delta_2 <1$ and a real-valued function on $\mathbb{R}^{n}\times \mathbb{R}^{m}$, the notation $C^{k_1+\delta_1, k_2+\delta_2}_{b} $ denotes (i) for all $\beta$ and $\gamma$ satisfying $0 \leq |\beta| \leq k_1$, $0 \leq |\gamma| \leq k_2$ and $|\beta|+|\gamma|\geq 1$, the partial derivative $\partial^{\beta}_{x}\partial^{\gamma}_{y}u$  is bounded continuous; (ii) $\partial^{\beta}_{x}\partial^{\gamma}_{y}u$ is $\delta_1$-H\"older continuous with respect to $x$ with index $\delta_1$ uniformly in $y$ and $\delta_2$-H\"older continuous with respect to $y$ with index $\delta_2$ uniformly in $x$.
\section{Preliminaries}
In this section, we recall some basic definitions for L\'evy motions.
\subsection{\textbf{Symmetric $\alpha $ -stable  process }}
A L\'evy process $L_t$ taking values in $\mathbb{R}^n$ is characterized by a drift vector $b \in {\mathbb{R}^n}$, an $n \times n$ non-negative-definite, symmetric covariance matrix $ Q $ and a Borel measure $\nu$ defined on ${\mathbb{R}^n}\backslash \{ 0\} $.   We call $(b,Q,\nu)$ the generating triplet of  the L\'evy motions $L_t$. Moreover, we have the L\'evy-It\^o decomposition for $L_t$ as follows
\begin{equation}
{L_t}= bt + B_{Q}(t) + \int_{||y||< 1} y \widetilde N(t,dy) + \int_{||y||\ge 1} y N(t,dy),
\end{equation}
where $N(dt,dy)$ is the Poisson random measure, $\widetilde N(dt,dy) = N(dt,dy) - \nu (dx)dt$ is the compensated Poisson random measure, $\nu (A) = \mathbb{E}N(1,A)$ is the jump measure, and $ B_{Q}(t)$ is an independent standard $n$-dimensional Brownian motion. The characteristic function of $L_t$ is given by
\begin{equation}
\mathbb{E}[\exp({\rm i}\langle u, L_t \rangle)]=\exp(t\rho(u)), ~~~u \in {\mathbb{R}^n},
\end{equation}
where the function $\rho:{\mathbb{R}^n}\rightarrow \mathbb{C}$ is  the characteristic exponent
\begin{equation}
\rho(u)={\rm i}\langle u, b\rangle-\frac{1}{2}\langle u, Qu \rangle+\int_{{\mathbb{R}^n}\backslash \{ 0\}}{(e^{{\rm i}\langle u, z \rangle}-1-{\rm i}\langle u,z\rangle {I_{\{ || z || \textless 1\} }})\nu(dz)}.
\end{equation}
The Borel measure $\nu $ is called the jump measure.

The following definition about rotationally symmetric $\alpha$-stable process comes from \cite[Definition 7.23]{duan}.
\begin{definition}
For $\alpha \in (0,2)$, an $n$-dimensional symmetric $\alpha $-stable process $ L^{\alpha}_{t} $ is a L\'evy process with characteristic exponent $\rho$
\begin{equation}
\rho(u)=-| u |^{\alpha},  ~for~u \in {\mathbb{R}^{n}}
\end{equation}
\end{definition}

For a $n$-dimensional symmetric $\alpha$-stable L\'evy process, the diffusion matrix $ Q= 0$,
the drift vector $ b= 0$, and the L\'evy measure $\nu $ is given by
\begin{equation}
\nu(du)=\frac{c(n,\alpha)}{{| u |}^{n+\alpha}}du,
\end{equation}
where $ c(n, \alpha):=\frac{\alpha\Gamma (\frac{n+\alpha}{2})}{2^{1-\alpha}\pi^{\frac{n}{2}}\Gamma(1-\frac{\alpha}{2})}$.

Let $(P_t)_{t\geq 0}$ be a semigroup of bounded linear operators on Banach space $\mathcal{B}_{b}(\mathbb{R}^{d})$. Let $\mu$ be a probability measure on Borel space $(\mathbb{R}^{d}, \mathcal{B}(\mathbb{R}^{d}))$. We use the following standard notation:
\begin{equation}
\langle\mu,\varphi \rangle=\int_{\mathbb{R}^{d}} \varphi(x)\mu(dx).
\end{equation}
$\mu$ is said to be an invariant probability measure of $P_t$ if
\begin{equation}
\langle \mu, P_t\varphi\rangle=\langle \mu, \varphi\rangle, \forall t>0, ~~\forall \varphi \in \mathcal{B}_{b}(\mathbb{R}^{d}).
\end{equation}
One says that $P_t$ is ergodic if $P_t$ admits a unique invariant probability measure $\mu$, which amounts to say that
\begin{equation}
\lim_{t\rightarrow \infty}\frac{1}{t}\int ^{t}_{0}P_s f(x)ds =\langle \mu, f\rangle, ~~f \in \mathbb{B}_{b}(\mathbb{R}^{d}).
\end{equation}
 The following definition gives the more precise classification about the ergodic \cite[Definition 2.7 ]{zhx}.
\begin{definition}
Let $\mathbb{V}: \mathbb{R}^{d}\rightarrow [1, \infty)$ be a measurable function and $\mu$ an invariant probability measure of $P_t$. We say $P_t$ to be $\mathbb{V}$-uniformly exponential ergodic if there exist $c_0, \gamma >0$ such that
\begin{equation}
\sup_{||\varphi||_{\mathbb{V}}\leq 1}|P_t\varphi(x)-\langle \mu, \varphi\rangle|\leq c_0 \mathbb{V}(x)e^{-\gamma t},
\end{equation}
where $||\varphi||_{\mathbb{V}}=\sup_{x\in \mathbb{R}^{d}}|\varphi(x)|< + \infty$. If $\mathbb{V}\equiv 1$, then $P_t$ is said to be uniformly  exponential ergodic, which is equivalent to
\begin{equation}
||P_t(x,\cdot)-\mu||_{Var}\leq c_0e^{-\gamma t}, \forall x \in \mathbb{R}^{d}.
\end{equation}
where $P_t(x, \cdot)$ is the kernel of bounded linear operator $P_t$.
\end{definition}

\section{Strong convergence analysis of the projective integration method}
\subsection{Stochastic averaging principle}
Consider the following singularly perturbed systems of stochastic differential equations of the form
\begin{equation}
\label{sys01}
\left\{
\begin{aligned}
dX^{\varepsilon}_t&=f_1(X^{\varepsilon}_t,Y^{\varepsilon}_t)dt +\sigma_1dL^{1}_t, ~~X^{\varepsilon}_{0}=x_0 \in \mathbb{R}^n,\\
dY^{\varepsilon}_t&=\frac{1}{\varepsilon}f_2(X^{\varepsilon}_t,Y^{\varepsilon}_t)dt+\frac{\sigma_2}{\varepsilon^{\frac{1}{\alpha}}}dL^{2}_t,
~~Y^{\varepsilon}_{0}=y_0 \in \mathbb{R}^m,
\end{aligned}
\right.
\end{equation}
where $L^{1}_t, L^{2}_t$ are
independent $n$ and $m$ dimensional symmetric  $\alpha$-stable processes with triplets $(0, 0, \nu)$. The function $f_1: \mathbb{R}^{n} \times \mathbb{R}^{m}\rightarrow \mathbb{R}^{n}$ and $f_2: \mathbb{R}^{n} \times \mathbb{R}^{m}\rightarrow \mathbb{R}^{m}$ are Borel functions.  The positive constants $\sigma_1$ and $\sigma_2$ represent the noises intensities. The parameter $\varepsilon$ describes the ratio of the time scale between the slow component and fast component.

We make the following assumptions for the slow-fast stochastic dynamical system \eqref{sys01}.

{\bf Hypothesis  H.1 }
The functions $f_1 \in C^{1+\gamma, 2+\delta}_{b}$ and $f_2 \in C^{1+\gamma, 2+\gamma}_{b}$ with some $\gamma \in (\alpha-1,1)$ and $\delta \in (0,1)$.

\par
{\bf Hypothesis  H.2 }
The function $f_2$ satisfies
\begin{equation}
\sup_{x \in \mathbb{R}^{n}}|f_2(x,0)|< \infty.
\end{equation}

\par
{\bf Hypothesis  H.3 }
There exists a positive constants $\beta$ such that for any $x \in \mathbb{R}^{n}, y_1,y_2 \in \mathbb{R}^{m}$,
\begin{equation}
\left\langle f_2(x,y_1)- f_2(x,y_2), y_1-y_2\right\rangle \leq -\beta| y_1-y_2 |^{2},
\end{equation}

\begin{remark}
Note that with the help of Hypothesis $ \bf{H.1}$,  there exist  positive constants $L$ and $K$ such that
\begin{equation}
\begin{aligned}
| f_1(x_1,y_1)-f_1(x_2,y_2)|&\leq L  \left(| x_1-x_2|+| y_1-y_2|\right), \\
| f_2(x_1,y_1)-f_2(x_2,y_2)|&\leq L \left(| x_1-x_2|+| y_1-y_2|\right),
\end{aligned}
\end{equation}
and
\begin{equation*}
| f_i(x,y)|\leq K(1+| x| +| y|),
\end{equation*}
for all $x_i, x\in \mathbb{R}^{n}, ~y_i, y\in \mathbb{R}^{m}, \; i=1, 2.$%
\end{remark}

\begin{remark}
Under Hypothesis $ \bf{H.3}$, the fast component $Y^{\varepsilon}_t$ ensures the existence of an invariant measure  $\mu_{x}(dy)$.
\end{remark}

The following result concerning the strong convergence for system \eqref{sys} was proved in \cite[Theorem 2.1]{XS1}.

\begin{lemma}[Strong convergenc]
\label{sce}
Under Hypotheses $\bf{H.1}$-$\bf{H.3}$, for any initial value $(x,y)\in \mathbb{R}^{n}\times \mathbb{R}^{m}$, $t\in[0,T]$ and $p \in (1, \alpha)$, we have
\begin{equation}
\mathbb{E}\left(\sup_{t \in [0, T]}|X^{\varepsilon}_t-\bar{X}_t|^p\right) \leq C \varepsilon^{p(1-1/\alpha)},
\end{equation}
where the effective equation is of the form
\begin{equation}
\label{effective}
d\bar{X}_t=\bar{f}_1(X_t)dt+\sigma_1dL^{\alpha}_{t}
\end{equation}
with
\begin{equation}
\bar{f}_1=\int_{\mathbb{R}^{m}}f_1(x,y)\mu_{x}(dy).
\end{equation}
\end{lemma}

However, it is often impractical to obtain the reduce equations in closed form, since the invariant measure is often unknown. Standard computational schemes may fails due to the separation between the $O(\varepsilon)$ time scale and the $O(1)$. This inspires us to develop a new algorithms to estimate the effect that the fast components have on slow ones.

\subsection{Numerical method}

For $n=1,2, \cdots, \lfloor T/\triangle \rfloor$, we assume that the slow component of \eqref{sys01} has the numerical solution $X_n$. The projective integration method consists of  a macro-solver:  an Euler-Maruyama time-stepper,
\begin{equation}
\label{org}
X_{n+1}=X_n+A(X_n)\triangle t+\sigma_1 \triangle L^{\alpha}_{n},
\end{equation}
where
\begin{equation}
\triangle L^{\alpha}_{n}=L^{\alpha}_{t_{n+1}}-L^{\alpha}_{t_{n}}.
\end{equation}
Given the coarse variable at the $n$-th  time step $X_n$, we assume that $Y^{n}_{m}, ~m=0,1, \cdots, M$ is the discrete variables associated with the fast dymanics  at the $n$-th coarse step, which are numerically generated by the Euler-Maruyama scheme with the time step $\delta t$ ($0 < \delta t \ll 1$), i.e.,
\begin{equation}
\label{micro}
Y^{n}_{m+1}=Y^{n}_{m}+\frac{1}{\varepsilon}f_2(X_n, Y^{n}_{m} ){\delta t}+\frac{\sigma_2}{\varepsilon^{\frac{1}{\alpha}}}\triangle L^{\alpha}_{m}, ~~Y^{n}_{0}=y_0,
\end{equation}
where
\begin{equation}
\triangle L^{\alpha}_{m}=L^{\alpha}_{t_{m+1}}-L^{\alpha}_{t_{m}},
\end{equation}
\begin{remark}
The function $A(X_n)$ is the approximation of $\bar{f}_1(X_{n})$.  We refer to \eqref{org} as the macro-solver. The sequence $Y^{n}_{m}$ is called the micro-solver. Equations \eqref{org} and \eqref{micro} define the projective integration method.
\end{remark}

Let $\triangle t$ be a fixed time step, and $\bar{X}_{n}$ be the numerical approximation to the coarse variable $\bar{X}$, at time $t_n=n\triangle t$. Inspired by the effective equation \eqref{effective}, $\bar{X}_{n}$ is evolved in time by an Euler-Maruyama step,
\begin{equation}
\label{macro}
\bar{X}_{n+1}=\bar{X}_n+\bar{f}_1(\bar{X}_n)\triangle t +\sigma_1 \triangle L^{\alpha}_{n},
\end{equation}
where $\triangle L^{\alpha}_{n}$ is $\alpha$-stable displacements over a time interval $\triangle t$.

Indeed, for every $c>0$, $L^{\alpha}_{ct}$ and $c^{\frac{1}{\alpha}}L^{\alpha}_t$ have the same distribution, then we easily gain the following lemma.
\begin{lemma}[Scaling transform]
\label{dis}
Let $Y^{\varepsilon}_t$ be the solution of the equation
\begin{equation}
dY^{\varepsilon}_t=\frac{1}{\varepsilon}f_2(x,Y^{\varepsilon}_t)dt+\frac{\sigma_2}{\varepsilon^{\frac{1}{\alpha}}}dL^{\alpha}_t,
\end{equation}
then ${Y_t}=Y^{\varepsilon}_{\varepsilon t}$ is a solution of the stochastic differential equation
\begin{equation}
d{Y_t}=f_2(x,{Y_t})dt+\sigma_2d\widehat{L}^{\alpha}_t,
\end{equation}
where $\widehat{L}^{\alpha}_t=\frac{1}{\varepsilon^{\frac{1}{\alpha}}}L^{\alpha}_{\varepsilon t}$.
\end{lemma}

By Lemma \ref{dis}, we know that the micro-solver \eqref{micro} is a particular realization that uses an Euler-Maruyama time-stepper as well, i.e.,
\begin{equation}
\label{ymicro}
{Y}^{n}_{ m+1}={Y}^{n}_{ m}+f_2(X_n, {Y}^{n}_{ m} )\delta t+\sigma_2\triangle {\widehat{L}}^{\alpha,n}_{m}.
\end{equation}
Thus $A(X_n)$ can be estimated by an empirical averaging
\begin{equation}
\label{An}
A(X_n)=\frac{1}{M}\sum_{m=1}^{M}f_1(X_n, Y^{n}_{m}).
\end{equation}

In the following, we will give a discrete version of Gronwall inequality.

Let $u_n$ and $\omega_n$ be nonnegative sequences, and $c$ a nonnegative constant. If
\begin{equation}
u_{n}\leq \sum_{l=0}^{ n-1}\omega_lu_l+c,
\end{equation}
then we have
\begin{equation}
u_n\leq c e^{\sum_{l=0}^{n-1}\omega_l}.
\end{equation}

Before proceeding the strong convergence of projective integration method for slow-fast stochastic dynamical systems  under $\alpha$-stable noises, we need to provide some estimates for the processes $Y^{n}_{m}$ and $X_n$.

\begin{lemma}[Estimate for microsolver]
\label{y}
For small enough $\delta t$ and $1<p< \alpha$, we have
\begin{equation}
\sup_{\mbox{\tiny$\begin{array}{c}
0\leq n \leq \lfloor \frac{T}{\triangle t}\rfloor\\
0\leq m \leq M\\
\end{array}$}} \mathbb{E}|Y^{n}_{m}|^{p} \leq C \left(\delta t\right)^{\frac{p}{\alpha}}, ~~Y^{n}_{0}=Y^{n-1}_{m}.
\end{equation}
\end{lemma}

\begin{proof}
By \eqref{ymicro}, Hypothesis $\bf{H.1}$ and Hypothesis $\bf{H.2}$,  we have
\begin{equation}
\begin{aligned}
\mathbb{E}\left|Y^{n}_{m+1}\right|^{p}&\leq C_{p}\mathbb{E}\left|Y^{n}_{m}\right|^{p}+C_{p}\mathbb{E}\left|f_2(X_n, Y^{n}_{m} )\delta t\right|^{p}+C_{p}\sigma_2^{p}\mathbb{E}|\triangle \widehat{L}^{\alpha,n}_{m}|^{p}\\
& = C_{p}\mathbb{E}\left|Y^{n}_{m}\right|^{p}+C_{p}\mathbb{E}\left|f_2(X_n, Y^{n}_{m} )\delta t\right|^{p}+C_{p,\sigma_2}\mathbb{E}| \widehat{L}^{\alpha,n}_{\delta t}|^{p}\\
& \leq  C_{p}\mathbb{E}\left|Y^{n}_{m}\right|^{p}+C_{p, L} \mathbb{E}\left|Y^{n}_{m}\right|^{p}(\delta t)^{p}+C_{p,L}\mathbb{E}|f_2(X_n,0)|^{p}(\delta t)^{p}+C_{p,\sigma_2}\mathbb{E}| \widehat{L}^{\alpha,n}_{\delta t}|^{p}\\
& \leq C_{p, L}\left(1+(\delta t)^{p}\right)\mathbb{E}\left|Y^{n}_{m}\right|^{p}+C_{p,L}(\delta t)^{p}+C_{ p,\sigma_2}\left({\delta t}\right)^{\frac{p}{\alpha}}.
\end{aligned}
\end{equation}
By the discrete Gronwall inequality, we have
\begin{equation}
\mathbb{E}\left|Y^{n}_{m}\right|^{p} \leq C \left(\delta t\right)^{\frac{p}{\alpha}}.
\end{equation}
\end{proof}

\begin{lemma}[Estimate for macrosolver]
\label{xn}
For the small enough $\Delta t <1$ and $1<p< \alpha$, we have
\begin{equation}
\sup_{0\leq n \leq \lfloor T/\Delta t\rfloor }\mathbb{E}|X_n|^{p}\leq C|\Delta t|^{\frac{p}{\alpha}}.
\end{equation}
\end{lemma}
\begin{proof}
By \eqref{org}, Hypothesis  H.1, Lemma \ref{y} and $\delta t \ll \Delta t$, we have
\begin{equation}
\begin{aligned}
\mathbb{E}|X_{n+1}|^{p}&\leq C_{p}\mathbb{E}|X_{n}|^{p}+C_{p}\mathbb{E}|A(X_n)|^{p}|\Delta t|^{p}+C_{p, \sigma_1}|\Delta t|^{\frac{p}{\alpha}}\\
& \leq C_{p}\mathbb{E}|X_{n}|^{p}+C_{p, M}(\Delta t)^p \sum_{m=1}^{M}\mathbb{E}|f_1(X_n, Y^{n}_{m})|^p+C_{p, \sigma_1}|\Delta t|^{\frac{p}{\alpha}}\\
& \leq C_{p}\mathbb{E}|X_{n}|^{p}+C_{p, M}(\Delta t)^p \sum_{m=1}^{M}\mathbb{E}|f_1(X_n, Y^{n}_{m})-f_1(0,Y^{n}_{m})|^p\\
&+C_{p, M}(\Delta t)^p \sum_{m=1}^{M}\mathbb{E}|f_1(0,Y^{n}_{m})|^p+C_{p, \sigma_1}|\Delta t|^{\frac{p}{\alpha}}\\
& \leq C_{p}\mathbb{E}|X_{n}|^{p}+C_{p, M}\mathbb{E}|X_{n}|^{p}(\Delta t)^p+C_{p, M, K}(\Delta t)^p \left(1+\mathbb{E}|Y^{n}_{m}|^p\right)+C_{p, \sigma_1}|\Delta t|^{\frac{p}{\alpha}}\\
& \leq C_{p}\mathbb{E}|X_{n}|^{p}+C_{p, M}\mathbb{E}|X_{n}|^{p}(\Delta t)^p+C_{p, M, K}(\Delta t)^p \left(1+C \left(\delta t\right)^{\frac{p}{\alpha}}\right)+C_{p, \sigma_1}|\Delta t|^{\frac{p}{\alpha}}\\
& \leq C_{p, M}\left(1+(\Delta t)^p\right)\mathbb{E}|X_{n}|^{p}+C_{p,\sigma_1, M, K}|\Delta t|^{\frac{p}{\alpha}}\\
& \leq C_{p, M}\mathbb{E}|X_{n}|^{p}+C_{p,\sigma_1, M, K}|\Delta t|^{\frac{p}{\alpha}}\\
\end{aligned}
\end{equation}
By the discrete Gronwall inequality, we have
\begin{equation}
\sup_{0\leq n \leq \lfloor T/\Delta t\rfloor }\mathbb{E}|X_{n}|^{p}\leq C|\Delta t|^{\frac{p}{\alpha}}.
\end{equation}
\end{proof}

\begin{lemma}[Estimate for the successive iterations of the microsolver]
For the small enough $\delta t$ and $1 <p< \alpha$, the deviation between two successive iterations of the microsolver satisfies
\begin{equation}
\sup_{\mbox{\tiny$\begin{array}{c}
0\leq n \leq \lfloor \frac{T}{\triangle t}\rfloor\\
0\leq m \leq M\\
\end{array}$}} \mathbb{E}|Y^{n}_{m+1}-Y^{n}_{m}|^{p} \leq C (\delta t)^{\frac{p}{\alpha}}.
\end{equation}
\end{lemma}
\begin{proof}
By Lemma \ref{y}, Lemma \ref{xn}  and Hypothesis $\bf{H.1}$, we have
\begin{equation}
\begin{aligned}
 \mathbb{E}|Y^{n}_{m+1}-Y^{n}_{m}|^{p}&\leq C_{p}\mathbb{E}\left| f_2(X_n, Y^{n}_{m})\right|^{p}(\delta t)^{p}+C_p \sigma_2^{p} \mathbb{E}|L^{\alpha, n}_{\delta t}|^{p}\\
&\leq  C_{p, K} \left(1+\mathbb{E}|X_n|^{p}+\mathbb{E}|Y^{n}_{m}|^{p} \right)(\delta t)^{p}+C_{p, \sigma_2}(\delta t)^{p/\alpha}\\
& \leq C \left(\delta t\right)^{\frac{p}{\alpha}}.
\end{aligned}
\end{equation}
Therefore we have
\begin{equation}
\sup_{\mbox{\tiny$\begin{array}{c}
0\leq n \leq \lfloor \frac{T}{\triangle t}\rfloor\\
0\leq m \leq M\\
\end{array}$}} \mathbb{E}|Y^{n}_{m+1}-Y^{n}_{m}|^{p} \leq C (\delta t)^{\frac{p}{\alpha}}.
\end{equation}
\end{proof}

Define the following auxiliary  process $z^{k}_t$, which satisfies the following stochastic differential equation,
\begin{equation}
\label{znn}
dz^{n}_t=f_2(X_{n}, z^{n}_t)dt +\sigma_2 d{\widehat{L}}^{\alpha,n}_t, ~~z^{n}_0=y_0,
\end{equation}
where ${\widehat{L}}^{\alpha,n}_t$ is in independent of ${L}^{\alpha}_t$.

\begin{lemma}[Estimate for the auxiliary process]
\label{z}
Under Hypotheses $\bf{H.1}$-$\bf{H.3}$, the process $z^{n}_t$ satisfies
\begin{equation}
\sup_{0\leq t \leq T}\mathbb{E}|z^{n}_t|^p\leq C \left(1+|y_0|^{p}\right).
\end{equation}
\end{lemma}
\begin{proof}
By \eqref{znn} and Hypotheses $\bf{H.1}$, we have
\begin{equation}
|z^{n}_t|\leq |y_0|+L \int^{t}_{0}|z^{n}_s|ds +\int^{t}_{0}|f_2(X_n,0)|ds+\sigma_2 \left|\widehat{L}^{\alpha,n}_{t}\right|.
\end{equation}
This implies that
\begin{equation}
\begin{aligned}
\mathbb{E}|z^{n}_t|^{p}&\leq C_p |y_0|^{p}+C_{p, L, T }\int^{t}_{0}\mathbb{E}|z^{n}_s|^{p}ds +C_{p,T}\mathbb{E}|f_2(X_n,0)|^{p}+C_{p,\sigma_2}t^{\frac{p}{\alpha}}\\
& \leq  C_p |y_0|^{p}+C_{p, L, T }\int^{t}_{0}\mathbb{E}|z^{n}_s|^{p}ds +C_{p,T}\mathbb{E}|f_2(X_n,0)|^{p}+C_{p,\sigma_2, T}\\
\end{aligned}
\end{equation}
By Gronwall inequality, we have
\begin{equation}
\sup_{0\leq t \leq T} \mathbb{E}|z^{n}_t|^{p} \leq C \left(1+|y_0|^{p}\right).
\end{equation}
\end{proof}

The following results illustrate that the dynamic \eqref{znn} is exponential ergodicity with invariant measure $\mu^{X_n}$, which comes from \cite[Proposition 3.5]{XS1}

\begin{lemma}[Exponential ergodicity property]
\label{avv}
Under Hypotheses $\bf{H.1}$-$\bf{H.3}$, there exists a positive constant $C$ such that for each fixed $X_{n}$ and $F\in C^{1}_{b}(\mathbb{R}^{n})$, we have
\begin{equation}
\begin{aligned}
\left|\mathbb{E}\left[F(z^{n}_{t})\right]-\int_{\mathbb{R}^{n}}F(y)\mu^{X_{n}}(dy)\right|\leq C\left(1+|y_0|\right)e^{-\beta t}.
\end{aligned}
\end{equation}
\end{lemma}

The next lemma establishes the mixing properties of the auxiliary process $z^{n}_{t}$.

\begin{lemma}[Mixing property]
\label{fbar}
Under Hypotheses $\bf{H.1}$-$\bf{H.3}$, for the small enough $\delta t$ and $1< p < \alpha $, we have
\begin{equation}
\left(\mathbb{E}\left|\frac{1}{M}\sum_{m=1}^{M}f_1(X_n, z^{n}_{m})-\bar{f}_1(X_n)\right|^{p}\right)^{\frac{1}{p}}\leq 2\sqrt{\frac{\ln(M\delta t)+\beta}{M\beta\delta t}}+\sqrt{\frac{1}{M}}.
\end{equation}
\end{lemma}
\begin{proof}
By the property of expectation, we have
\begin{equation}
\begin{aligned}
&\left(\mathbb{E}\left|\frac{1}{M}\sum_{m=1}^{M}f_1(X_n, z^{n}_{m})-\bar{f}_1(X_n)\right|^{p}\right)^{\frac{1}{p}}\\
&\leq \left(\mathbb{E}\left|\frac{1}{M}\sum_{m=1}^{M}f_1(X_n, z^{n}_{m})-\bar{f}_1(X_n)\right|^{2}\right)^{\frac{1}{2}}\\
& \leq \frac{1}{M}\left(\sum_{m=1}^{M}\sum_{l=1}^{M}\left\{\mathbb{E}\left[f_1(X_n, z^{n}_{m})-\bar{f}_1(X_n)\right]\cdot \left[f_1(X_n, z^{n}_{l})-\bar{f}_1(X_n)\right]\right\}\right)^{\frac{1}{2}}\\
& \leq \frac{\sqrt{2}}{M}\left(\sum_{m=1}^{M}\sum_{l=m+1}^{M}\left\{\mathbb{E}\left[f_1(X_n, z^{n}_{m})-\bar{f}_1(X_n)\right]\cdot \left[f_1(X_n, z^{n}_{l})-\bar{f}_1(X_n)\right]\right\}\right)^{\frac{1}{2}}\\
&+\frac{1}{M}\left(\sum_{m=1}^{M}\left\{\mathbb{E}\left[f_1(X_n, z^{n}_{m})-\bar{f}_1(X_n)\right]\cdot \left[f_1(X_n, z^{n}_{m})-\bar{f}_1(X_n)\right]\right\}\right)^{\frac{1}{2}}\\
&:=\frac{\sqrt{2}}{M}\sqrt{J_1}+\frac{1}{M}\sqrt{J_2}.
\end{aligned}
\end{equation}
For $J_1$, by the Markov property and Hypotheses $\bf{H.2}$, we have
\begin{equation}
\begin{aligned}
& \sum_{m=1}^{M}\sum_{l=m+1}^{M}\{\mathbb{E}\left[f_1(X_n, z^{n}_{m})-\bar{f}_1(X_n)\right]\cdot \left[f_1(X_n, z^{n}_{l})-\bar{f}_1(X_n)\right]\}\\
& \leq \sum_{m=1}^{M}\sum_{l=m+1}^{M}\mathbb{E}\{\left[f_1(X_n, z^{n}_{m})-\bar{f}_1(X_n)\right]\cdot \mathbb{E}_{z^{n}_{m}}\left[f_1(X_n, z^{n}_{l-m})-\bar{f}_1(X_n)\right]\}\\
& \leq \sum_{m=1}^{M}\sum_{l=m+1}^{m+N} e^{-\beta(l-m)\delta t}+\sum_{m=1}^{M}\sum_{l=m+N+1}^{M} e^{-\beta(l-m)\delta t}\\
& \leq  M N+M^2e^{-\beta N\delta t}.
\end{aligned}
\end{equation}
Set $N=\frac{\ln(M \delta t)}{\beta \delta t}$, then we have
\begin{equation}
\label{j1}
 \sum_{m=1}^{M}\sum_{l=m+1}^{M}\left\{\mathbb{E}\left[f_1(X_n, z^{n}_{m})-\bar{f}_1(X_n)\right]\cdot
 \left[f_1(X_n, z^{n}_{l})-\bar{f}_1(X_n)\right]\right\} \leq \frac{ M ln(M \delta t)}{\beta \delta t}+\frac{M}{\delta t}.
\end{equation}
Similarly, for $J_2$, we have
\begin{equation}
\label{j2}
J_2 \leq \sum_{m=1}^{M}e^{-\beta m\delta t} \leq M.
\end{equation}
Combined with \eqref{j1} and \eqref{j2}, we have
\begin{equation}
\left(\mathbb{E}\left|\frac{1}{M}\sum_{m=1}^{M}f_1(X_n, z^{n}_{m})-\bar{f}_1(X_n)\right|^{p}\right)^{\frac{1}{p}} \leq
2\sqrt{\frac{\ln(M\delta t)+\beta}{M\beta\delta t}}+\sqrt{\frac{1}{M}}
\end{equation}
\end{proof}

Next we will establish the deviation between \eqref{znn} and its numerical approximation \eqref{ymicro}.

\begin{lemma}[Deviation between the microsolver and auxiliary process]
\label{lemma7}
Let $z^{n}_{t}$ be the family of process defined by \eqref{znn}. For small enough $\delta t$ and $1< p < \alpha$, we have
\begin{equation}
\max_{0 \leq n \leq \lfloor \frac{T}{\triangle t}\rfloor}\mathbb{E}|Y^{n}_{m}-z^{n}_{m}|^{p} \leq C (\delta t)^{\frac{p}{\alpha}}
\end{equation}
\end{lemma}
\begin{proof}
Set
\begin{equation}
Y^{n}_{t}=Y^{n}_{0}+\int^{t}_{0}f_2(X_{n}, Y^n_{ \lfloor s/\delta t\rfloor\delta t})ds +\sigma_2 {\widehat{L}}^{\alpha,n}_t.
\end{equation}
Then $Y^{n}_{t}$ is the Euler-Maruyama approximation of $Y^{n}_{m}$.

Define $v_t=Y^{n}_{t}-z^{n}_{t}$, then we have
\begin{equation}
\begin{aligned}
\mathbb{E}|v_t|^{p}&\leq \int ^{t}_{0}\mathbb{E}\left|f_2(X_{n}, Y^n_{ \lfloor s/\delta t\rfloor\delta t})-f_2(X_{n},z^{n}_s)\right|^p ds\\
&\leq C_{p,T}\int ^{t}_{0}\mathbb{E}\left|f_2(X_{n}, Y^n_{ \lfloor s/\delta t\rfloor\delta t})-f_2(X_{n},Y^{n}_s)\right|^p ds+C_p\int ^{t}_{0}\mathbb{E}\left|f_2(X_{n}, Y^{n}_s)-f_2(X_{n},z^{n}_s)\right|^p ds \\
&\leq C_{p, L, T}\int ^{t}_{0}\mathbb{E}\left| Y^{n}_{ \lfloor s/\delta t\rfloor\delta t}-Y^{n}_s\right|^p ds +C_{p, L} \int^{t}_{0}\mathbb{E}|v_s|^{p}ds.
\end{aligned}
\end{equation}
By Gronwall's inequality, we have
\begin{equation}
\label{vt}
\mathbb{E}|v_t|^{p}\leq C_{p, L, T} \int ^{t}_{0}\mathbb{E}\left| Y^{n}_{ \lfloor s/\delta t\rfloor\delta t}-Y^{n}_s\right|^p ds.
\end{equation}
Using Lemma \ref{y} and Lemma \ref{xn}, we have
\begin{equation}
\label{ych}
\begin{aligned}
\mathbb{E}\left| Y^{n}_t-Y^{n}_{ \lfloor t/\delta t\rfloor\delta t}\right|^p &\leq C_{p,T}\int^{t}_{\lfloor t/\delta t\rfloor\delta t}\mathbb{E}
\left|f_2\left(X_{n}, Y^{n}_{\lfloor s/\delta t\rfloor\delta t}\right)\right|^{p}ds +C_{p, \sigma_2}\mathbb{E}\left|L^{\alpha,n}_{t-\lfloor t/\delta t\rfloor\delta t}\right|^p\\
& \leq C_{p,T}\int^{t}_{\lfloor t/\delta t\rfloor\delta t}\mathbb{E}
\left|f_2\left(X_{n}, Y^{n}_{\lfloor s/\delta t\rfloor\delta t}\right)\right|^{p}ds +C_{p, \sigma_2}|\delta t|^{\frac{p}{\alpha}}\\
& \leq C_{p,T}\int^{t}_{\lfloor t/\delta t\rfloor\delta t} \left(1+\mathbb{E}|X_n|^{p}+\mathbb{E}|Y^{n}_{\lfloor s/\delta t\rfloor\delta t}|^p\right)ds+C_{p, \sigma_2}|\delta t|^{\frac{p}{\alpha}}\\
& \leq C\left(1+(\delta t)^{\frac{p}{\alpha}}+(\Delta t)^{\frac{p}{\alpha_1}}\right)\delta t+C_{p, \sigma_2}|\delta t|^{\frac{p}{\alpha}}\\
& \leq C (\delta t)^{\frac{p}{\alpha}}.
\end{aligned}
\end{equation}
Take \eqref{ych} into \eqref{vt}, we have
\begin{equation}
\max_{0 \leq n \leq \lfloor \frac{T}{\triangle t}\rfloor}\mathbb{E}|Y^{n}_{m}-z^{n}_{m}|^{p} \leq C (\delta t)^{\frac{p}{\alpha}}.
\end{equation}
\end{proof}
In the following, we will give the approximation estimate between $A(X_{n})$ and $\bar{f}_{1}(X_n)$.
\begin{lemma}[Approximation estimate for the drift coefficient]
\label{er}
Under Hypotheses $\bf{H.1}$-$\bf{H.3}$, for any $0\leq n \leq \lfloor T/\triangle t\rfloor$ and $1< p < \alpha$, we have
\begin{equation}
\mathbb{E}\left|A(X_{n})-\bar{f}_{1}(X_n)\right|^{p} \leq C\left((\delta t)^{\frac{p}{\alpha}}+\left(\frac{\ln(M\delta t)+\beta }{M\beta\delta t}\right)^{\frac{p}{2}}+\left(\frac{1}{M}\right)^{\frac{p}{2}}\right).
\end{equation}
\end{lemma}
\begin{proof}
By definition  of $A(X_{n})$, Lemma \ref{fbar} and Lemma \ref{lemma7}, we have
\begin{equation}
\begin{aligned}
&\mathbb{E}\left|A(X_{n})-\bar{f}_{1}(X_n)\right|^{p} \\
&\leq C_{p}\mathbb{E}\left|\frac{1}{M}\sum_{m=1}^{M}f_1(X_n, Y^{n}_{m})-\frac{1}{M}\sum_{m=1}^{M}{f}_{1}(X_n, z^{n}_{m})\right|^{p} +C_{p}\mathbb{E}\left|\frac{1}{M}\sum_{m=1}^{M}{f}_{1}(X_n, z^{n}_{m})-\bar{f}_{1}(X_n)\right|^{p}\\
& \leq C_{p, L}\max_{m\leq M}\mathbb{E}\left|Y^{n}_{m}-z^{n}_{m}\right|^{p}+C_{p}\left(2\sqrt{\frac{\ln(M\delta t)+\beta}{M\beta\delta t}}+\sqrt{\frac{1}{M}}\right)^{p}\\
& \leq C(\delta t)^{\frac{p}{\alpha}}+C_{p}\left(2\sqrt{\frac{\ln(M\delta t)}{M\beta\delta t}+\frac{1}{M\delta t}}+\sqrt{\frac{1}{M}}\right)^{p}\\
& \leq C(\delta t)^{\frac{p}{\alpha}}+C\left(\frac{\ln(M\delta t)+\beta }{M\beta\delta t}\right)^{\frac{p}{2}}+C\left(\frac{1}{M}\right)^{\frac{p}{2}} \\
& \leq C\left((\delta t)^{\frac{p}{\alpha}}+\left(\frac{\ln(M\delta t)+\beta }{M\beta\delta t}\right)^{\frac{p}{2}}+\left(\frac{1}{M}\right)^{\frac{p}{2}}\right).
 \end{aligned}
 \end{equation}
\end{proof}

The following lemma shows the function $\bar{f}_1$ has Lipschitz property.
\begin{lemma}[Lipschitz property]
\label{lipc}
Under Hypotheses $\bf{H.1}$-$\bf{H.3}$, the functions $\bar{f}_1$  satisfies Lipschitz condition.
\end{lemma}
\begin{proof}
For any $t>0$, $x_1, x_2 \in \mathbb{R}^{n}$, by Hypotheses $\bf{H.1}$ and Lemma \ref{avv}, we have
\begin{equation}
\label{0f}
\begin{aligned}
\left| \bar{f}_1(x_1)-\bar{f}_1(x_2)\right|&\leq \left|\bar{f}_1(x_1)-\mathbb{E}f_1(x_1, z^{n}_{s,x_1})\right|
+\left|\mathbb{E}f_1(x_2, z^{n}_{s,x_2})-\bar{f}_1(x_2)\right|\\
&+\left|\mathbb{E}f_1(x_1, z^{n}_{s,x_1})-f_1(x_2, z^{n}_{s,x_2})\right|\\
&\leq Ce^{-\beta t}+C\left(|x_1-x_2|+\mathbb{E}|z^{n}_{s,x_1}-z^{n}_{s,x_2}|\right)\\
& \leq Ce^{-\beta t}+C|x_1-x_2|.
\end{aligned}
\end{equation}
Letting $t\rightarrow \infty$, then we know the function $\bar{f}_1$ has the Lipschitz property.
\end{proof}

Next we will give the rate of strong convergence for the multiscale method.
\begin{theorem}
Under Hypotheses $\bf{H.1}$-$\bf{H.3}$, for any  $0\leq n \leq \lfloor T/\triangle t\rfloor$ and $1< p < \alpha$, we have
\begin{equation}
\sup_{0\leq n \leq \lfloor T/\triangle t\rfloor}\mathbb{E}\left|X_n-\bar{X}_n\right|^{p} \leq C\left(\left(\delta t\right)^{\frac{p}{\alpha}}+\left(\frac{\ln(M\delta t)+\beta }{M\beta\delta t}\right)^{\frac{p}{2}}+\left(\frac{1}{M}\right)^{\frac{p}{2}}\right),
\end{equation}
where $\delta t$ is the step size of the microsolver, and $M= \triangle t/{\delta t }$.
\end{theorem}
\begin{proof}
Set $E_n=\mathbb{E}|X_n-\bar{X}_n|^{p}$, by Lemma \ref{er} and Lemma \ref{lipc}, we have
\begin{equation}
\begin{aligned}
E_{n}&=\mathbb{E}\left|\sum_{i=0}^{n-1}\left[A(X_i)-\bar{f}_{1}(\bar{X}_{i})\right]\triangle t\right|^{p}\\
&\leq C_{p}\mathbb{E}\left|\sum_{i=0}^{n-1}\left[A(X_i)-\bar{f}_{1}({X}_{i})\right]\triangle t\right|^{p}+ C_{p}\mathbb{E}\left|\sum_{i=0}^{n-1}\left[\bar{f}_1(X_i)-\bar{f}_{1}(\bar{X}_{i})\right]\triangle t\right|^{p}\\
& \leq C_{p, T}\max_{i<n}\mathbb{E}\left|A(X_i)-\bar{f}_{1}({X}_{i})\right|^{p}+C_{p, \widetilde{L}}\sum_{i=0}^{n-1}E_{i}\left(\triangle t\right)^{p}\\
& \leq  C(\delta t)^{\frac{p}{\alpha}}+C\left(\frac{\ln(M\delta t)+\beta }{M\beta\delta t}\right)^{\frac{p}{2}}+C\left(\frac{1}{M}\right)^{\frac{p}{2}}+C_{p, \widetilde{L}}\sum_{i=0}^{n-1}E_{i}\left(\triangle t\right).
\end{aligned}
\end{equation}
By a discrete version of Gronwall inequality, we have
\begin{equation}
E_{n} \leq C\left((\delta t)^{\frac{p}{\alpha}}+\left(\frac{\ln(M\delta t)+\beta }{M\beta\delta t}\right)^{\frac{p}{2}}+\left(\frac{1}{M}\right)^{\frac{p}{2}}\right).
\end{equation}
Therefore we have
\begin{equation}
\sup_{0\leq n \leq \lfloor T/\triangle t\rfloor}\mathbb{E}\left|X_n-\bar{X}_n\right|^{p} \leq C\left(\left(\delta t\right)^{\frac{p}{\alpha}}+\left(\frac{\ln(M\delta t)+\beta }{M\beta\delta t}\right)^{\frac{p}{2}}+\left(\frac{1}{M}\right)^{\frac{p}{2}}\right).
\end{equation}
\end{proof}
\section{Weak convergence analysis of the projective integration method}
In this section, we will give the  weak convergence analysis of the projective integration method, and the $p$th moment error bounds between the results of the projective integration method and the slow components of the original system  with $p \in \left(1,  \alpha\right)$.

 The following rate of weak convergence for the two-time scale stochastic dynamical systems driven by $\alpha$-stable processes was proved in  \cite[Theorem 2.3]{XS1}.

\begin{lemma}[Weak convergence]
\label{weak}
Suppose that the assumptions in Lemma \ref{sce}  holds. Further assume that $f_1, f_2 \in C^{2+\gamma, 2+\gamma}_{b}$ with $\gamma \in \left(\alpha-1, 1\right)$. Then for any $\phi \in C^{2+\gamma}_{b}(\mathbb{R}^{m})$ and initial value $(x, y) \in \mathbb{R}^{n}\times \mathbb{R}^{m}$, we have
\begin{equation}
\sup_{t \in [0, T]}\left|\mathbb{E}\phi(X^{\varepsilon}_t)-\mathbb{E}\phi(\bar{X}_t)\right|\leq C \varepsilon.
\end{equation}
where $C$ is a positive constant depending on $T$, $||\phi||_{C^{2+\gamma}_{b}}$, $|x|$ and $|y|$, and $\bar{X}_t$ is the solution of the averaged equation \eqref{effective}.
\end{lemma}

Next we will give the rate of weak convergence for the PIM.
\begin{theorem}[Weak convergence for the PIM  ]
Let $X_n$ be the Euler approximation for $X_t$ and $\bar{X}_{n}$  be the Euler approximation for $\bar{X}_t$, then for any $\phi \in C^{2+\gamma}_{b}(\mathbb{R}^{m})$, we have
\begin{equation}
\sup_{0\leq n \leq \lfloor T/\triangle t\rfloor}\left|\mathbb{E}\phi(X_n)-\mathbb{E}\phi(\bar{X}_n)\right|\leq  C\left((\delta t)^{\frac{p}{\alpha}}+\left(\frac{\ln(M\delta t)+\beta }{M\beta\delta t}\right)^{\frac{p}{2}}+\left(\frac{1}{M}\right)^{\frac{p}{2}}\right)^{\frac{1}{p}}\triangle t.
\end{equation}
\end{theorem}
\begin{proof}
For any $n \leq \lfloor T/\triangle t\rfloor$, we construct the following  auxiliary function $u(k,x_0)$, i.e.,
\begin{equation}
u(k,x_0)=
\begin{cases}
&\phi(x_0), ~~k=n, \\
&\mathbb{E}\left[u(k+1, x_0+\bar{f}_1(x_0)\Delta t+\sigma_1\Delta L^{\alpha}_{\cdot})\right], ~~k<n,
\end{cases}
\end{equation}
then we have
\begin{equation}
u(0,x_0)=\mathbb{E}\phi(\bar{X}_{n}).
\end{equation}
By the smoothness of $\phi$, it is easy to show that  $\sup_{k,x}\left|\frac{\partial u(k,x)}{\partial x}\right|$ is uniformly bounded. Therefore we have
\begin{equation}
\begin{aligned}
&\left|\mathbb{E}\phi(X_n)-\mathbb{E}\phi(\bar{X}_n)\right|\\
&=\left|\mathbb{E}u(n, X_{n})-u(0,x_0)\right|\\
&=\left|\mathbb{E}\left(\sum_{l=0}^{n-1}u(l+1, X_{l+1})-u(l,X_{l})\right)\right|\\
&=\left|\sum_{l=0}^{n-1}\mathbb{E}\left(u(l+1, X_{l+1})-u(l+1, X_l)-\left(u\left(l+1,\bar{X}^{l, X_{l}}_{l+1}\right)-u\left(l+1,X_l\right)\right)\right)\right|\\
&\leq \sum_{l=0}^{n-1}\mathbb{E}\left\{\sup\left|\frac{\partial u}{\partial x}\right|\left|\left(X_{l+1}-X_{l}-\left(\bar{X}^{l, X_{l}}_{l+1}-X_{l}\right)\right)\right|\right\}\\
& \leq C\sum_{l=0}^{n-1} \triangle t \mathbb{E}|A(X_l)-\bar{f}_1(X_l)|.
\end{aligned}
\end{equation}
By Lemma \ref{er}, we have
\begin{equation}
\sup_{0\leq n \leq \lfloor T/\triangle t\rfloor}|\mathbb{E}\phi(X_n)-\mathbb{E}\phi(\bar{X}_n)|\leq C\left((\delta t)^{\frac{p}{\alpha}}+\left(\frac{\ln(M\delta t)+\beta }{M\beta\delta t}\right)^{\frac{p}{2}}+\left(\frac{1}{M}\right)^{\frac{p}{2}}\right)^{\frac{1}{p}}\triangle t.
\end{equation}
\end{proof}
\section{Numerical Example}
\begin{example}
Consider the following slow-fast stochastic dynamical systems
\begin{equation}
\label{ex}
\left\{
\begin{array}{rl}
dX^{\varepsilon}_t&=-X^{\varepsilon}_tdt+\sin(X^{\varepsilon}_t)e^{-\left(Y^{\varepsilon}_t\right)^2}dt+dL_t^{{\alpha }},\\
dY^{\varepsilon}_t&=-\frac{Y^{\varepsilon}_t}{\e}dt+ \frac{1}{\e^{\frac{1}{\alpha }}}dL_t^{{\alpha }}.
\end{array}
\right.
\end{equation}
where $f_1(x,y)=-x+\sin x e^{-y^2}$,  $f_2(x,y)=-y$ and $\sigma_1=\sigma_2=1$. It is easy to justify that $f_1, f_2$ satisfy Hypotheses $\bf{H.1}$-$\bf{H.3}$.
Using a result in \cite{ar}, we find the   invariant measure   $\mu(dx)=\rho(x)dx$ with density
\begin{equation}
\rho(x)=\frac{1}{{2\pi}}\int_{\mathbb{R}}e^{ix\xi}e^{-\frac{1}{\alpha}|\xi |^{\alpha}}d\xi.
\end{equation}
Then the effective equation for $X^{\varepsilon}_t$ is
\begin{equation}
d\bar{X}(t)=-\bar{X}(t)dt+ \bar{a} \sin(\bar{X}(t))dt+dL_t^{{\alpha }},
\end{equation}
 where
\begin{equation*}
\bar{a}=\frac{\alpha^{\frac{1}{\alpha}}}{2\sqrt{2}}\int_{\mathbb{R}}e^{-\frac{1}{4}\xi^2-\frac{1}{\alpha}| \xi|^{\alpha}}d\xi.\\
\end{equation*}
\end{example}

The numerical study is performed for the above model. In Fig.~\ref{eps},  we take $x_0=10, y_0=10$, $\alpha=1.5$, $ M=100, N=1000$ with different $\varepsilon$, The macro and micro time steps are $\Delta t=0.001$, $\delta t= \Delta t/M$, respectively. The results, presented  in Fig.~\ref{eps}, indicate that the solution of slow component produced by the PIM approximates the solution of effective equation. It is worth emphasizing that even if the invariant measures cannot be obtained explicitly, we can still detect slow variables by PIM method. In Fig.~\ref{error}, we compute the $L^p$ error between $X_n$ and $\bar{X}_n$ with $p=1.2$ for different $\varepsilon$, where $L^p$  error $= \sum_{k=1}^l|X_n^k- \bar{X}_n^k|^p/l$. As seen in Fig.~\ref{error} , the $L^p$  error is larger as the $\varepsilon$ becomes larger.

\begin{figure}
\centering
\subfigure[$\varepsilon=0.1$]{
\begin{minipage}[t]{0.5\linewidth}
\centering
\includegraphics[height=5cm]{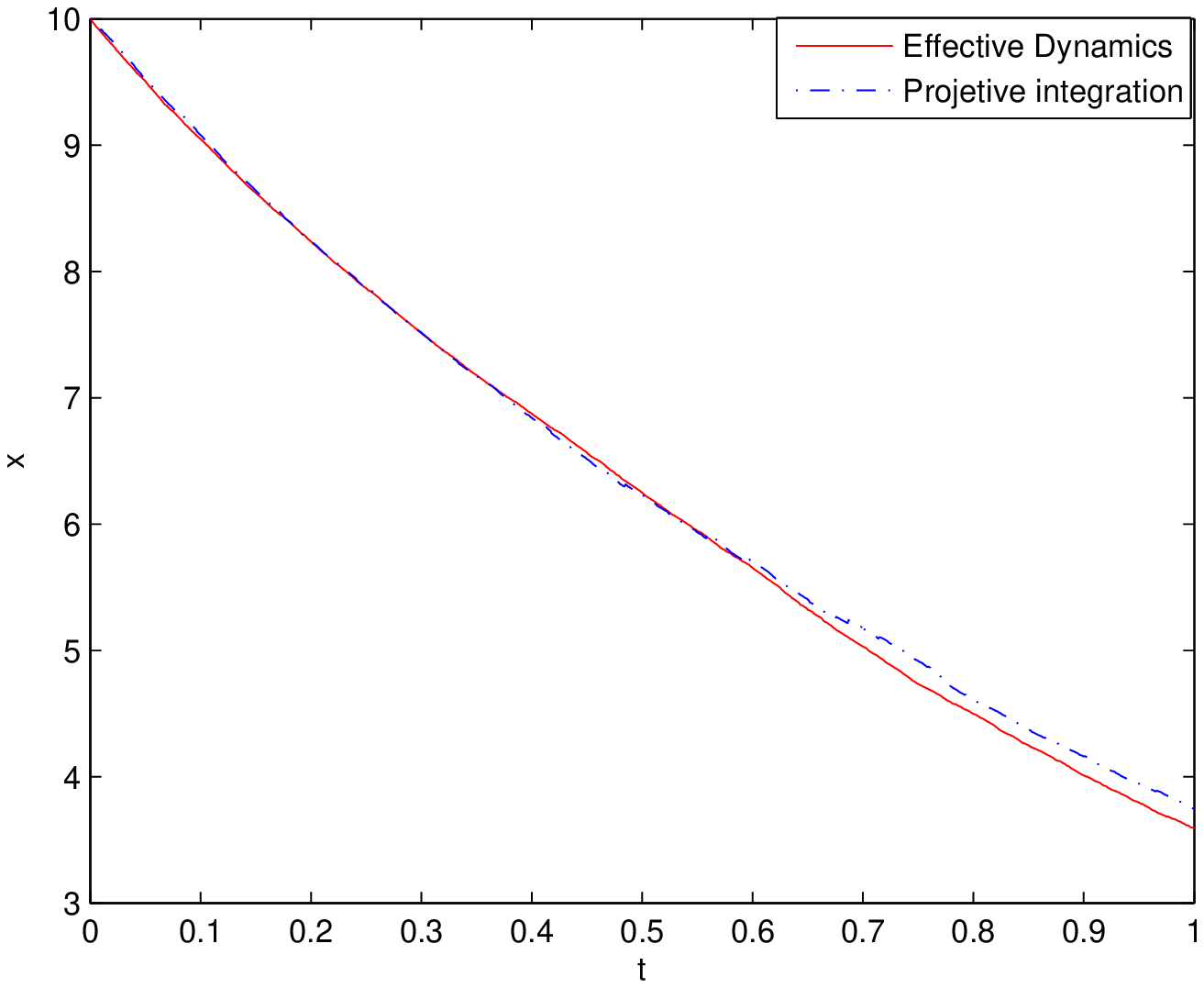}
\end{minipage}%
}%
\subfigure[$\varepsilon=0.01$]{
\begin{minipage}[t]{0.5\linewidth}
\centering
\includegraphics[height=5cm]{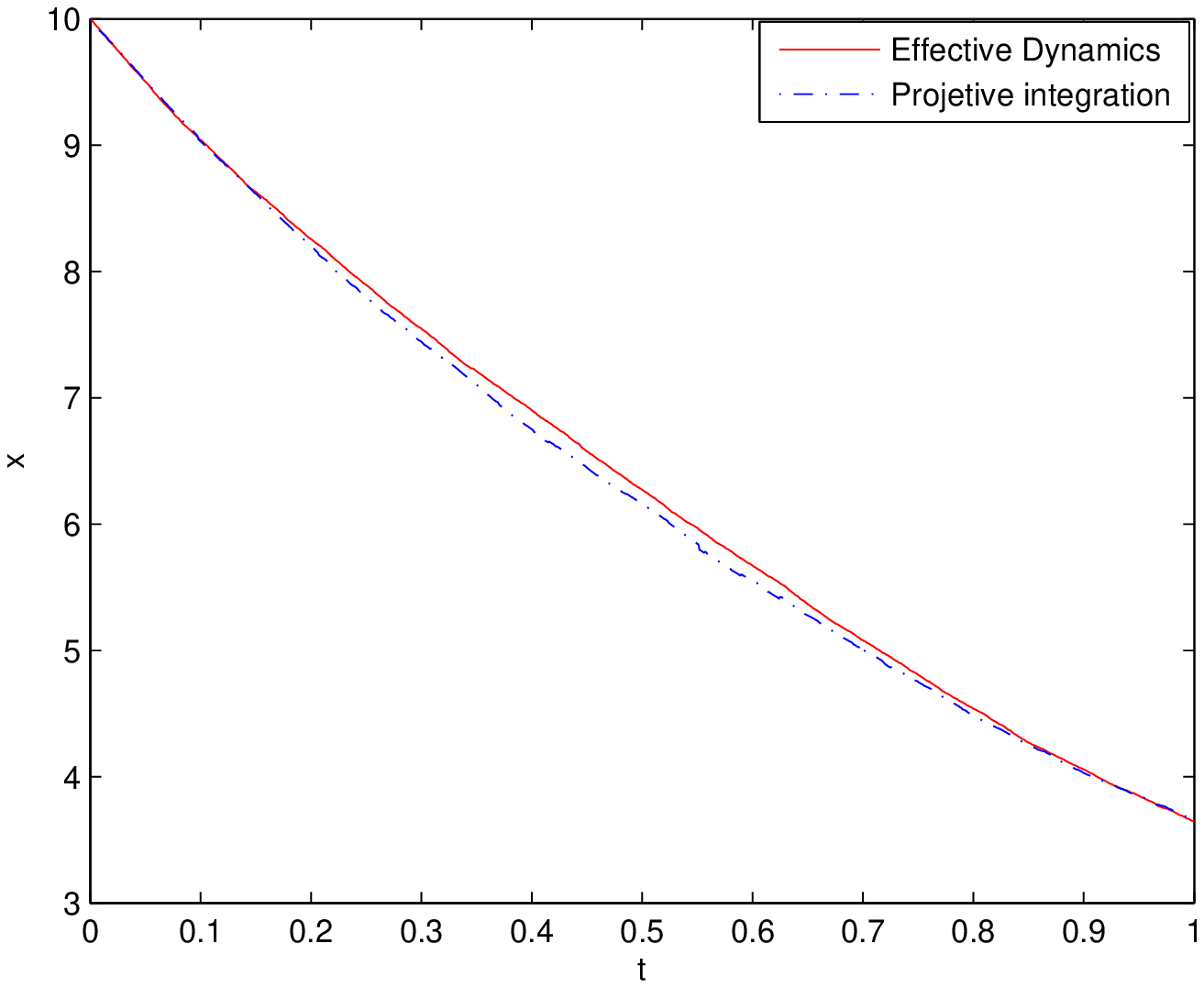}
\end{minipage}%
}%
\centering
\caption{ Compare the effective dynamics with the stochastic projective integration schemes.}
\label{eps}
\end{figure}

\begin{figure}
\centering
\subfigure[$\varepsilon=0.1$]{
\begin{minipage}[t]{0.5\linewidth}
\centering
\includegraphics[height=5cm]{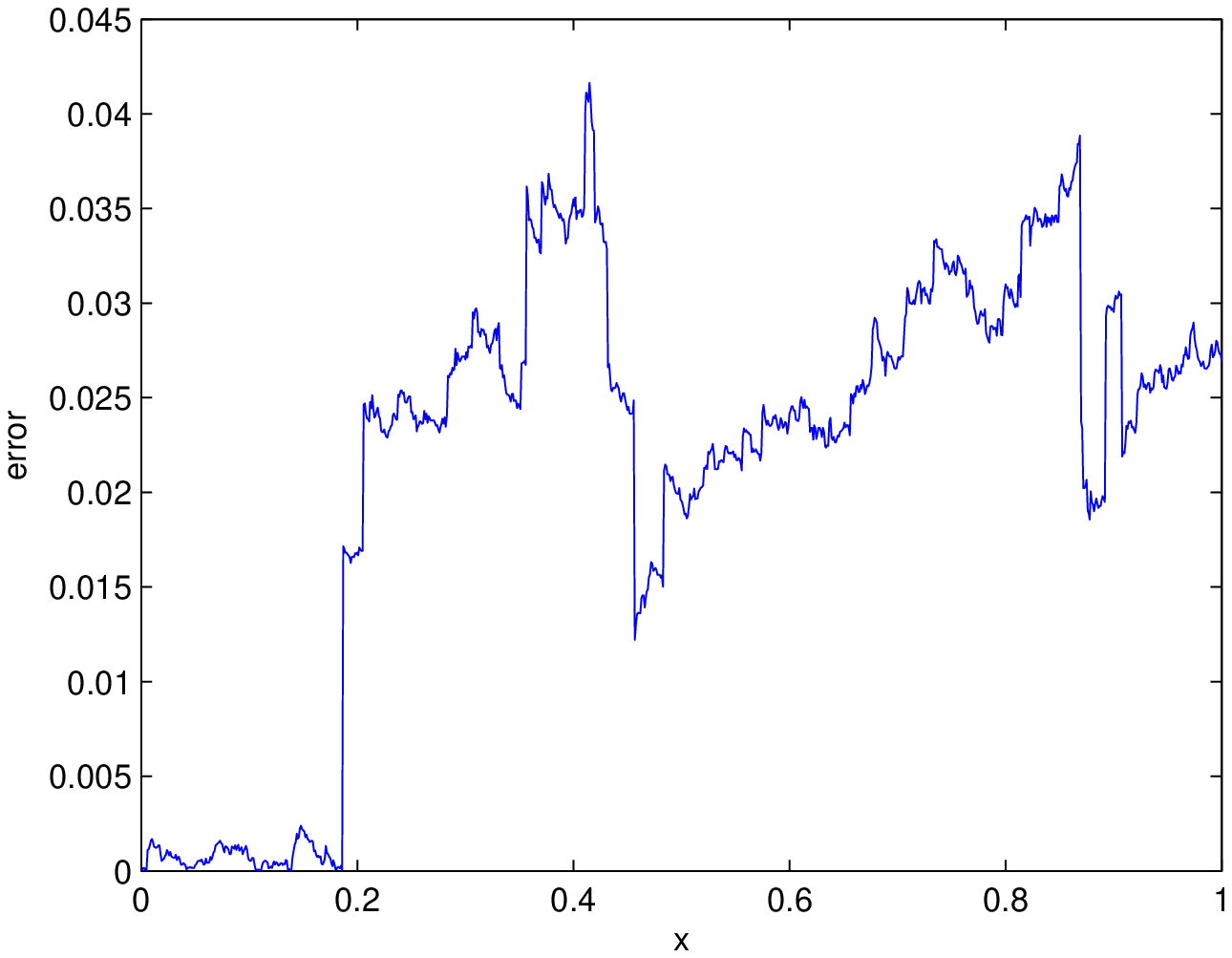}
\end{minipage}%
}%
\subfigure[$\varepsilon=0.01$]{
\begin{minipage}[t]{0.5\linewidth}
\centering
\includegraphics[height=5cm]{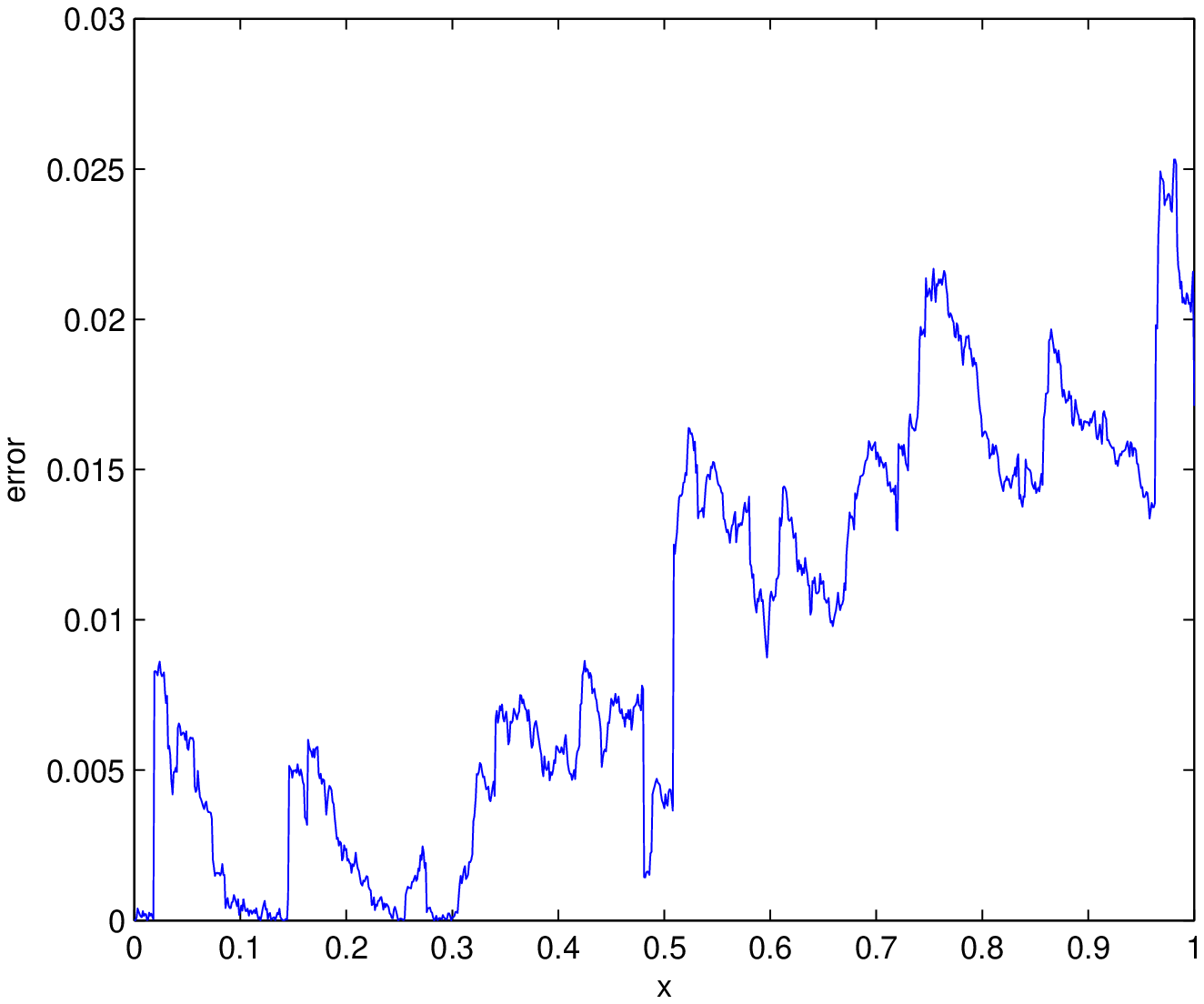}
\end{minipage}%
}%
\centering
\caption{ The $L^{p}$ error between $X_n$ and $\bar{X}_n$ for $\varepsilon=0.1$.}
\label{error}
\end{figure}

\begin{table}[tbp]
\caption{the error $E_p$ for different $l$. }
\begin{center}
{\setlength{\tabcolsep}{8mm}
 \begin{tabular}{cccccc}
 \hline

 \hline

 \hline
 l    &  1   & 2 & 3    & 4  & 5\\

\hline

 p=1.4 & 0.3324  &  0.1711   & 0.0759 &  0.0444  &    0.0270    \\
\hline

\hline

\end{tabular}
}
\end{center}
\label{error11}
 \end{table}

 \begin{figure}
\begin{center}
\includegraphics*[width=0.6\linewidth]{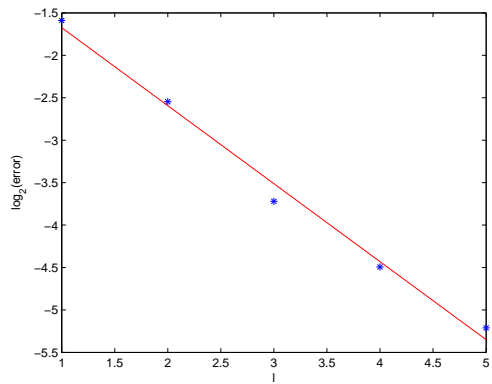}
\end{center}
\caption{The logarithm error $\log_2(E_p)$ with $p=1.4$, $\varepsilon=0.1$ . The
solid line is a line with $slope=0.92$. }
\label{orderError}
\end{figure}

Suppose that $T=1$ is fixed,  we
want to bound the error by $\mathbb{O}(2^{-l})$ for $l=1,2,...$, then the
optimal choice is to take $\delta t=\mathcal {O}(2^{-\frac{l}{p}\alpha_2})$  and
$M=\mathcal {O}(2^{(2+\alpha_2)\frac{l}{p}})$ by simple analysis. Here we compute the following error
estimate averaged with the $K=1000$ sample paths of the macro dynamics
$$
  E_p=\frac{1}{K}\sum_{k=1}^K\left(\sum_{n\leq \frac{T}{\Delta}} \frac{\Delta}{T}|X_n^k- \bar{X}_n^k|^p\right).
$$
The errors $E_p$ for different $l$ are listed in Table \ref{error11}, and the logarithm errors $\log_2(E_p)$  are also
shown in Fig.~\ref{orderError}. We observe that the errors are agree with our theoretical result $\mathbb{O}(2^{-l})$ .

\section{Discussion}
In this paper, we studied the strong and weak convergences of projective integration methods for the  multiscale stochastic dynamical systems driven by $\alpha$-stable processes, which were used to estimate the effect that the fast components have on slow ones. Moreover, we obtained the $p$th moment error bounds between the solution of slow component produced by the PIM and the solution of effective dynamical system with $p \in \left(1,  \alpha\right)$. Our scheme provided a new numerical method to obtain the slow component even if  the invariant measure was unknown.

By reexamining  the argument of the above results, the parameter value $\alpha$ of noise terms in slow component and fast component can be different, for example
\begin{equation}
\label{sys01}
\left\{
\begin{aligned}
dX^{\varepsilon}_t&=f_1(X^{\varepsilon}_t,Y^{\varepsilon}_t)dt +\sigma_1dL^{\alpha_1}_t, ~~X^{\varepsilon}_{0}=x_0 \in \mathbb{R}^n,\\
dY^{\varepsilon}_t&=\frac{1}{\varepsilon}f_2(X^{\varepsilon}_t,Y^{\varepsilon}_t)dt+\frac{\sigma_2}{\varepsilon^{\frac{1}{\alpha}}}dL^{\alpha_2}_t,
~~Y^{\varepsilon}_{0}=y_0 \in \mathbb{R}^m,
\end{aligned}
\right.
\end{equation}
where $L^{\alpha_1}_t, L^{\alpha_2}_t$ are
independent $n$ and $m$ dimensional symmetric  $\alpha$-stable processes with triplets $(0, 0, \nu_1)$ and $(0, 0, \nu_2)$, respectively.

There are some limitations for this paper. The condition $p \in (1, \alpha)$ plays an important role in deriving the effective dynamical system. Therefore, for the technical reason, it seems hard to show Lemma  \ref{sce} and Lemma \ref{weak} with $p \in (0, 1)$. For such a case, it is necessary to find some new approaches to study. It is an open problem for the convergence in probability of averaging principle for the systems  \eqref{sys}.

As an extension of the model, we also will  study equations like \eqref{sys} on the diffusive time scale, i.e.,
\begin{equation}
\label{sys}
\left\{
\begin{aligned}
dX^{\varepsilon}_t&=f_1(X^{\varepsilon}_t,Y^{\varepsilon}_t)dt++ \sigma_1 dL_t^{{\alpha}_1}, ~~X^{\varepsilon}_{0}=x_0 \in \mathbb{R}^n,\\
dY^{\varepsilon}_t&=\frac{1}{\varepsilon^2}f_2(X^{\varepsilon}_t,Y^{\varepsilon}_t)dt+\frac{\sigma_2}{\varepsilon^{\frac{2}{\alpha}}}dL^{\alpha_2}_t,
~~Y^{\varepsilon}_{0}=y_0 \in \mathbb{R}^m,
\end{aligned}
\right.
\end{equation}
Some new theoretical and numerical results will be reported in the forthcoming manuscript.

\medskip
\textbf{Acknowledgements}.
 The research of Y. Zhang was supported by the NSFC grant 11901202.  The research
of X. Wang was supported by the NSFC grant 11901159. The research of J. Duan was supported  by the NSF-DMS no. 1620449 and NSFC grant. 11531006 and  11771449.

\end{document}